\documentclass[a4paper,fleqn]{cas-dc}
\usepackage{vmargin}
\setmarginsrb{1.27cm}{1.5cm}{1.27cm}{1.5cm}{12pt}{12pt}{12pt}{12pt}

\usepackage{algorithmic}
\usepackage{algorithm}
\usepackage{bookmark}
\usepackage{graphicx}
\usepackage{psfrag}
\usepackage{subcaption}
\usepackage{amsmath}
\usepackage{amsthm}
\usepackage{caption}
\usepackage{float}
\usepackage{booktabs}
\usepackage{caption}
\usepackage{adjustbox}
\usepackage{hyperref}
\usepackage{xcolor}

\usepackage{dblfnote}
\usepackage[authoryear,colon,sort&compress]{natbib}

\hypersetup{
  colorlinks=true,
  linkcolor=black,
  urlcolor=black,
  citecolor=black,
  linktoc=all
}

\def\tsc#1{\csdef{#1}{\textsc{\lowercase{#1}}\xspace}}
\tsc{WGM}
\tsc{QE}
\tsc{EP}
\tsc{PMS}
\tsc{BEC}
\tsc{DE}

\newtheoremstyle{notationstyle}
  {}{}
  {}
  {}
  {\bfseries}
  {.}
  { }
  {}
\theoremstyle{notationstyle}
\newtheorem*{notation}{Notation}
 \newtheorem{theorem}{Theorem}
 \newtheorem{lemma}[theorem]{Lemma}
 \newdefinition{rmk}{Remark}
\newdefinition{assumption}{Assumption}
\newdefinition{problem}{Problem}

\makeatletter
\renewenvironment{proof}[1][\proofname]{\par
	\pushQED{\qed}
	\normalfont \topsep6\p@\@plus6\p@\relax
	\trivlist
	\item[\hskip\labelsep
	\bfseries
	#1\@addpunct{.}]\ignorespaces
}{%
	\popQED\endtrivlist\@endpefalse
}
\makeatother

\begin{document}
\let\WriteBookmarks\relax
\def\floatpagepagefraction{1}
\def\textpagefraction{.001}
\let\printorcid\relax

\shorttitle{SPI Algorithm for DT linear systems}

\shortauthors{Z. Pang,  S. Tang, C. Jun, S. He
}

\title [mode = title]{
	Scaling policy iteration based  reinforcement learning for unknown  discrete-time linear systems}   
	\tnotemark[1]     


\author[1]{Zhen Pang}
  \ead{ zhenpang@stu.gxnu.edu.cn
    }

    \author[1]{Shengda Tang}
    \ead{ tangsd911@163.com}
    \cormark[1]
    
    \author[1]{Jun Cheng}[]
    \ead{jcheng@gxnu.edu.cn}
    
    \author[2,3,4]{Shuping He}[]
    \ead{shuping.he@ahu.edu.cn}

    \affiliation[1]{organization={ School of Mathematics and Statistics},
    addressline={Guangxi Normal University},
    city={Guilin 541004},
    country={PR China}}
    \affiliation[2]{organization={Information Materials and Intelligent Sensing
    Laboratory  of Anhui Province },
    addressline={School of Electrical Engineering and
    Automation, Anhui University},
    city={Hefei 230601},
    country={PR China}} 

    \affiliation[3]{organization={School
    of Electronic Information and Electrical Engineering },
    addressline={ Chengdu University},
    city={Chengdu 610106},
    country={PR China}}
    \affiliation[4]{organization={Institute of Artificial Intelligence},
    addressline={ 
    Hefei Comprehensive National Science Center },
    city={Hefei 230088},
    country={PR China}}

\credit{Data curation, Writing - Original draft preparation}

\cortext[1]{Corresponding author} 
\tnotetext[1]{ The work described in this paper was supported by the National Natural Science Foundation of China (No. 61761008), Science and Technology Project of Guangxi (Guike
AD21220114).
}

\begin{abstract}
  In optimal control problem,  policy iteration (PI) is a powerful reinforcement learning (RL) tool used for designing optimal controller for the linear systems.  
   However, the need for an initial stabilizing control policy significantly limits its applicability.  
  	%
	  To address this constraint, this paper proposes a novel scaling technique, which progressively brings a sequence of stable scaled systems closer to the original system, enabling the acquisition of stable control gain.
  %
  Based on the designed scaling update law, we develop model-based and model-free scaling policy iteration (SPI) algorithms for solving the optimal control problem for discrete-time linear systems, in both known and completely unknown system dynamics scenarios. Unlike existing works on PI based RL, the SPI algorithms do not necessitate  an initial stabilizing gain to initialize the algorithms,  they can achieve the optimal control under any initial control gain.
  Finally, the numerical results validate the theoretical findings and confirm the effectiveness of the algorithms.
  
  \end{abstract}
  \maketitle
\section{Introduction}\label{Sec_1}
In recent years, system stability and optimal control have remained key areas of focus in control theory \cite{rigatos_2017nonlinear}.  The optimal control problem centers on identifying the control input that best enables a system to achieve its predefined objective. By leveraging optimization algorithms, these control inputs allow for the precise regulation of complex systems, such as autonomous vehicles, industrial processes, and robotics, while taking into account factors like energy efficiency, safety, and stability \cite{2023_virtual}. 
A large number of numerical methods have emerged to resolve the optimal control problem, most of which are based on policy iteration (PI) due to its quadratic convergence rate.
It is worth noting that  PI has a wide range of applications, such as in model-based stochastic system control problems   and infinite-horizon discrete-time (DT) Markov decision problems with a discount factor \cite{bertsekas_1996_MDP,bertsekas_2019_MDP,winnicki_2023_MDP}.

To overcome the requirement for system information, reinforcement learning (RL) has also been introduced \cite{zamfirache_2023neural}.
Among the methods for solving optimal control problem, PI-based RL is one of the most effective approaches.
This method demonstrates tremendous capability and convenience, and 
it is commonly utilized to address optimization problems with system dynamics that are either model-free or partially model-free \cite{ RL_2,lai2023model}. 
%
As an example, \cite{track} utilized an off-policy RL algorithm to solve an optimal output tracking problem, where the system dynamics were entirely unknown. 
In \cite{li2022stochastic},  a partially model-free policy algorithm was introduced to design the optimal controller for the stochastic continuous-time (CT) linear system. 
Additionally, \cite{zero_sum_1_DT} presented a model-free RL algorithm that utilizes data collected throughout the system trajectories to tackle the zero-sum game problem for  DT systems.
For these majority of existing PI-based RL algorithms, an initial stabilizing control policy is required to initiate these algorithms, which maybe unavailable in some cases.
Generally, the design of a stabilizing control policy requires a priori knowledge of the system dynamics.
However, as modern engineering systems grow in scale and complexity, discerning their precise dynamics becomes an increasingly daunting task. 
Consequently, how to eliminate the reliance on  initial stable control for PI  has become  a critical issue. This motivated our research.

Some relevant research results have emerged in recent years.
Value iteration (VI) is widely recognized as a significant method for solving optimal control problems. While VI can start with any initial control input, it typically requires more iterations than PI. 
By combining the strengths of the VI and PI algorithms, generalized policy iteration (GPI) algorithm \cite{lee2014integral,ZhouGpI} and hybrid iteration (HI) algorithm \cite{HI2023adaptive,HICT,qasem2023hybrid,DT_HIFUZZY}  were developed, and both of these algorithms have effectively eliminated the reliance on  an  initial stabilizing control policy.  
In  \cite{lamperski2020computing,lai2023learning}, a discount factor-based method is proposed to attain optimal control solution without requiring an initial stabilizing control. 
Additionally, in \cite{de2019formulas,van2020data}, data-based methods are introduced for designing stabilizing control gain for DT systems by solving linear matrix inequality (LMI) or a system of equations. Subsequently, \cite{RL_2} builds on the approach from \cite{van2020data} by designing a deadbeat control gain matrix, which is then used to initialize a PI-based off-policy Q-learning algorithm, effectively addressing the linear quadratic regulator (LQR) problem. Which gives us a great inspiration.
  Furthermore,  for CT systems,
 \cite{chen2022homotopic}  has put forward a homotopy-based PI algorithm for CT optimal control problem, which a stabilizing control policy can be achieved by
 gradually pushing the stabilizing system to the original system.  
 As demonstrated in 
 \cite{wang2023,Homotopic_n},  the homotopy-based PI algorithm can be extensively employed to tackle various problems encountered in CT systems, such as $H_{\infty}$ control, the optimal control problems of CT Markovian jump systems and nonlinear systems, eliminating the requirement of an initial stabilizing control input. However, this algorithm is only applicable to CT systems and it cannot be directly parallelized to DT systems. 
 Therefore, whether it is possible to develop a similar approach for DT systems that eliminates the need for an initial stabilizing control strategy for PI is another motivation for our work.

 Motivated by the aforementioned works,  this paper introduces a novel technique to address the optimal control problem in DT linear systems. The main contributions of this work are as follows: Firstly,
 we propose an innovative scaling technique that can convert any initial control policy into a stabilizing one, offering a new perspective on solving optimal control problems. Secondly, based on this scaling technique, we developed both model-based and model-free scaling policy iteration (SPI) algorithms. These algorithms solve the optimal control problem for DT linear systems with known and completely unknown dynamics, respectively, starting from any initial control policy. Thirdly, the proposed scaling technique is highly versatile and can be flexibly applied to control problems in DT linear systems where $P_{i}$ exhibits monotonic non-decreasing behavior during PI.

  The remainder of this paper is structured as follows.
  Problem formulation and relevant preliminaries  are briefly introduced in Section \ref{Sec_2}.  A model-based algorithm is given in Section \ref{Sec_3}, and  a model-free algorithm is developed in Section \ref{Sec_4}.  In Section \ref{Sec_5},  a numerical example is presented,  and  Section \ref{sec_6} concludes this paper.  
  \begin{notation}
    Throughout this paper, the notation $\mathbb{R} ^{n}$ and $\mathbb{R} ^{n\times m}$  represent the set of real vectors with $n$ dimensions and the set of real matrices with dimensions $n\times m$, respectively. $||\cdot ||$ denotes the Euclidean norm for a vector or matrix of  appropriate size.  For a given matrix $Y\in\mathbb{R} ^{n\times m}$, we use $Y^{-1}$ and $Y^{T}$ to denote its inverse and transpose, respectively. 
   The symbol $\otimes$ stands  the Kronecker product, and $vec(Y)=[y_{1}^{T},y_{2}^{T},\dots,y_{m}^{T}]^{T}$, where $y_{i}\in \mathbb{R} ^{n} $ are the column vectors of $Y$. When the matrix $Y$ is a square matrix,  $\rho(Y)$ is employed to denote its spectral radius,  and symbol $\rho^{-1}(Y)$ stands the reciprocal of $\rho(Y)$.
   $\sigma(Y)$ and  $\sigma(Y)^{1/2}$ represent, respectively, the minimum singular value and the square root of the minimum singular value of matrix $Y$. $I_{n}$ denotes the identity matrix with dimensions $n \times n$. And zero vector or matrix is denoted by $0$. For a symmetric matrix $S$, $S>0$ (resp. $S\geqslant 0$) indicates that matrix $S$ is positive (resp. positive semidefinite). Specially, for $S \in \mathbb{R} ^{n\times n}$,
   define    $vecs(S)=[s_{11},2s_{12},\dots,2s_{1n},s_{22},2s_{23},\dots,2s_{(n-1,n)},s_{nn}]^{T}$, where $vecs(S) \in \mathbb{R} ^{\frac{1}{2}n\left( n+1 \right)}$. Similarly, given a vector $z \in \mathbb{R} ^{n}$, define $vecv(z)=[z_{1}^{2},z_{1}z_{2},\dots,z_{1}z_{n},z_{2}^{2},z_{2}z_{3},\dots,z_{n-1}z_{n},z_{n}^{2}]^{T}$.
  \end{notation}
  \section{Problem formulation and preliminaries}\label{Sec_2} 
In this section, we present a  description of the system model and present relevant preliminaries  of the current research, which helps us establish a solid foundation for our subsequent analysis.  


We consider the following  DT linear system given by
	\begin{align}\label{eq1}
		x_{k+1}=Ax_k+Bu_k,
     \end{align}
where $u_k\in \mathbb{R}^{m}$ and $x_k\in \mathbb{R} ^{n}$ are respectively the control input and the system state. The matrices $A \in \mathbb{R}^{n\times n}$ and $B \in \mathbb{R}^{n\times m}$ represent the state and input matrices, and  it should be noted that the matrix $A$ is not necessarily stable in this paper.

Define the associated 
 performance index as 
\begin{align} \label{eq2}
	     V\left( x_k \right)= \sum_{t=k}^{\infty}{\left[ x_{k}^{T}Qx_k+u_{k}^{T}Ru_k \right]},
\end{align}
where $Q=Q^{T}\geqslant 0$ and $R=R^{T}>0$ are the running weighting matrices
for the state and the control, respectively.
\begin{assumption}
	The  pair $(A,B)$ is  
	controllable, and the pair  $(A,\sqrt{Q})$ is observable.
\end{assumption}

The optimal control problem considered in this paper, i.e., the
  LQR problem, can be formulated as follows. 
 \begin{problem}
	Find an optimal control $u^{*}$ in  terms of  $u_{k} =-Kx_{k}$ to minimize (\ref{eq2}) and subject to (\ref{eq1}), where $K\in \mathbb{R} ^{m\times n}$ satisfies  $\rho \left( A-BK \right) <1$.
 \end{problem}

According to the  linear optimal control theory (\cite{lewis2012optimal}),   
  the optimal control policy  can be solved by
$	u^{\ast}_{k}  = -K^{\ast}x_k,$  with 
 \begin{align}\label{eq4}
 K^{\ast}= (R + B^{T}P^{\ast}B)^{-1}B^{T}P^{\ast}A,
 \end{align}  
 where the symmetric matrix $P^{\ast}= (P^{\ast})^{T} > 0$ is the unique solution of the following well-known DT algebraic Riccati equation (ARE)
 \begin{align}\label{eq6}
 	&A^{T}P^{\ast}A-P^{\ast}-A^{T}P^{\ast}B(R+B^{T}P^{\ast}B)^{-1}B^{T}P^{\ast}A\nonumber\\
 	&+Q=0.
 \end{align}

To solve the ARE (\ref{eq6}), the model-based PI algorithm, which forms the theoretical foundation for the development of our algorithms, was proposed by  \cite{hewer1971iterative} and is outlined in the following lemma.
	\begin{lemma} \label{lemma1}
		If $A-BK_{0}$ is  Schur stable,  which means that  $\rho(A-BK_{0})<1$. 
	 Let 
			\begin{align} \label{eq7}
			& P_i=Q+K_{i}^{T}RK_i+( A-BK_i ) ^TP_i(A-BK_i ),\\
			\label{eq8}
			&  K_{i+1}=( R+B^TP_iB) ^{-1}B^TP_iA,
		\end{align}
			for $i=0, 1, 2, \textcolor{blue}{\dots}$.
		Then, the following statements hold true:
		\begin{itemize}
			\item[(i)] $A-BK_{i}$ is Schur stable;
			\item[(ii)] $P^{*}\leqslant P_{i+1}\leqslant P_{i}$;
			\item[(iii)] $\underset{k\rightarrow \infty}{\lim}P_{i}=P^{*}$ , $\underset{i\rightarrow \infty}{\lim}K_{i}=K^{*}.$
		\end{itemize}
\end{lemma} 

	Based on Lemma \ref{lemma1}, the model-based PI algorithm requires an initial stabilizing control gain matrix. Consequently, the model-free PI algorithm is similarly constrained by this requirement. 	
	However, when the system information is completely unknown, obtaining a stabilizing control gain matrix is a daunting task. To overcome this difficulty, \cite{HI2023adaptive} and \cite{de2019formulas,van2020data,RL_2} respectively used VI and data-based methods to find the initial stabilizing control gain matrix. 
		Unlike DT systems, there is a homotopy-based algorithm for finding an initial stabilization policy in CT systems, which is completely dependent on the PI and is highly applicable. Therefore, inspired by the above works,  		
		this paper presents for the first time a PI-based scaling technique for DT linear systems, which transforms an arbitrarily given control gain matrix into a stabilizing control gain matrix from a novel perspective. 
	 And then, based on the scaling technique, we present a SPI algorithm in both the model-based and model-free scenarios to solve the optimal control problem. 
\section{Model-based SPI algorithm design}\label{Sec_3}
In this section, leveraging the known system matrices $A$ and $B$, we introduce a model-based SPI algorithm  to solve the ARE (\ref{eq6}) for the system (\ref{eq1}). The outcomes presented in this section will lay the groundwork for the design of model-free SPI algorithm in the subsequent  section.

 
To develop the SPI algorithm in accordance with the system dynamics, we first give the ensuing  Lemma \ref{lemma2}. 

	\begin{lemma}\label{lemma2}
	Given a control gain matrix $\widetilde{K}_0 $  arbitrarily,  let $c_{0}=1$, and select a constant $b$,  such that  
	 \begin{align}\label{eq9}
	b>\rho(A-B\widetilde{K}_{0}).
	\end{align}
	Then,  the subsequent statements are confirmed:
			\item[(i)]  The matrix $\frac{\prod_{j=0}^i{c_j}}{b}( A-B\widetilde{K}_i )$ is Schur stable for all $i = 0,  1, 2, \ldots$.
			\item[(ii)]	  	The following  Lyapunov equation
		\begin{align}\label{eq10}
		&\frac{\prod_{j=0}^i{c^{2}_j}}{b^2}( A-B\widetilde{K}_i ) ^T\widetilde{P}_{i}( A-B\widetilde{K}_i )-\widetilde{P}_{i}+Q\nonumber\\
			&+\widetilde{K}_{i}^{T}R\widetilde{K}_i=0
	\end{align}
	has an unique solution   $\widetilde{P}_i=\widetilde{P}^{T}_i>0$ 	for each $i=0,1,2,\ldots,$ where
		\begin{align}\label{eq11}
	&	\widetilde{K}_{i+1}=\Big( B^T\widetilde{P}_{i}B+\frac{b^2}{\prod_{j=0}^i{c^{2}_j}}R \Big) ^{-1}B^T\widetilde{P}_{i}A, \\
\label{eq12}
&			1< c_{i+1} <\rho^{-1}\Big (\frac{\prod_{j=0}^{i}{c_j}}{b}(A-B\widetilde{K}_{i+1})\Big).
	\end{align}
\end{lemma}
	
\begin{proof}  The lemma will be proven using induction, which is initialized by considering the case when $i = 0.$
		 Given $\widetilde{K}_0$, and noting $b>\rho(A-B\widetilde{K}_{0})$, $c_{0}=1$, one has
		$$	\frac{c_{0}}{b}	\rho ( A-B\widetilde{K}_0) =	\rho \left( \frac{1}{b}(A-B\widetilde{K}_0) \right) <1 .$$
		
		Then, the following Lyapunov equation	
		\begin{align*}
		&\frac{1}{b^2}( A-B\widetilde{K}_0) ^T\widetilde{P}_{0}( A-B\widetilde{K}_0 )-\widetilde{P}_{0}+Q+\widetilde{K}_{0}^{T}R\widetilde{K}_0=0
	\end{align*}
	admits  a unique positive definite symmetric matrix  solution $\widetilde{P}_0.$


Now, we suppose that the statements  hold for $i=\ell$. This leads to that
	there exists an unique positive definite symmetric  solution  $\widetilde{P}_\ell$  to
		\begin{align}\label{eq_P_l}
		&\frac{\prod_{j=0}^{\ell}{c^{2}_j}}{b^{2}}( A-B\widetilde{K}_{\ell} ) ^T\widetilde{P}_{\ell}( A-B\widetilde{K}_{\ell} ) -\widetilde{P}_{\ell}+Q
			+\widetilde{K}_{\ell}^{T}R\widetilde{K}_{\ell}=0,
	\end{align}
and from (\ref{eq11}) we have 
\begin{align}\label{eq_K_l+1}
	\widetilde{K}_{\ell+1}=\Big( B^T\widetilde{P}_{\ell}B+\frac{b^2}{\prod_{j=0}^{\ell}{c^{2}_j}}R \Big) ^{-1}B^T\widetilde{P}_{\ell}A.
\end{align} 
Now, we will show that if $c_{\ell+1}$ satisfies (\ref{eq12}) with $i=\ell$, then the statements hold for $i=\ell+1$. 

Consider the positive definite cost function $ V_{\ell}(x_k) = x_k^T \widetilde{P}_{\ell} x_k $ as a potential Lyapunov function for the state trajectories governed by the control input $ u_{k} =- \widetilde{K}_{\ell+1} x_{k} $, where $ x_{k+1} = \frac{\prod_{j=0}^{\ell} c_j}{b} (Ax_{k} + Bu_{k})$.
Taking the difference of $V_{\ell}( x_k )$ along the trajectory generated by $\widetilde{K}_{\ell+1}$ yields:
\begin{align*}
	\varDelta  V_{\ell}( x_k ) =&V_{\ell}( x_{k+1} ) -V_{\ell}( x_k ) 
\nonumber\\
=&x_{k}^{T}\Big\{ \frac{\prod_{j=0}^{\ell}{c_{j}^{2}}}{b^2}\Big[ ( A-B\widetilde{K}_{\ell} ) ^T\widetilde{P}_{\ell}( A-B\widetilde{K}_{\ell} ) 
\nonumber\\
	&+( \widetilde{K}_{\ell}-\widetilde{K}_{\ell +1} ) ^TB^T\widetilde{P}_{\ell}B( \widetilde{K}_{\ell}-\widetilde{K}_{\ell +1} ) 
	\nonumber\\
	&+2( \widetilde{K}_{\ell}-\widetilde{K}_{\ell +1} ) ^T\!B^T\widetilde{P}_{\ell}( A-B\widetilde{K}_{\ell} ) \Big] -\widetilde{P}_{\ell} \Big\} x_k.
\end{align*}
From (\ref{eq_K_l+1}) and performing some mathematical operations we get
\begin{align*}
	&2\frac{\prod_{j=0}^{\ell}{c_{j}^{2}}}{b^2}x_{k}^{T}( \widetilde{K}_{\ell}-\widetilde{K}_{\ell +1} ) ^TB^T\widetilde{P}_{\ell}( A-B\widetilde{K}_{\ell} )x_{k}
	\\
	=&x_{k}^{T}\Big[\!-2\frac{\prod_{j=0}^{\ell}{c_{j}^{2}}}{b^2}( \widetilde{K}_{\ell}\!-\widetilde{K}_{\ell +1} ) ^TB^T\widetilde{P}_{\ell}B( \widetilde{K}_{\ell}\!-\widetilde{K}_{\ell +1} )\!+\!\widetilde{K}_{_{\ell}}^{T}R\widetilde{K}_{\ell}
	\\
	& -\!( \widetilde{K}_{\ell}
	-\widetilde{K}_{\ell +1} ) ^T\!R( \widetilde{K}_{\ell}-\widetilde{K}_{\ell +1} ) -\widetilde{K}_{_{\ell +1}}^{T}R\widetilde{K}_{\ell +1}\Big]x_{k}.
\end{align*}
Bringing this  into the above equation and combining it with equation (\ref{eq_P_l}), one obtains
\begin{align*}
	\varDelta V_{\ell}( x_k ) =&-x_{k}^{T}\Big[Q+  ( \widetilde{K}_{\ell}-\widetilde{K}_{\ell +1} ) ^T(\frac{\prod_{j=0}^{\ell}{c_{j}^{2}}}{b^2}B^T\widetilde{P}_{\ell}B+R)\\
	&\times( \widetilde{K}_{\ell}-\widetilde{K}_{\ell +1} ) +\widetilde{K}_{_{\ell +1}}^{T}R\widetilde{K}_{\ell +1} \Big] x_k
	<0.
\end{align*}
Therefore, $V_{\ell}( x_k )$ is a Lyapunov function, one then has that

	$$\rho\Big( \frac{\prod_{j=0}^{\ell}{c_j}}{b}(A-B\widetilde{K}_{\ell+1}) \Big)<1.$$
Noting (\ref{eq12}), then one has
\begin{align}
	&\rho\Big(\frac{\prod_{j=0}^{\ell+1}{c_j}}{b}(A-B\widetilde{K}_{\ell+1})\Big)	\nonumber \\
	=&c_{\ell+1}\rho\Big(\frac{\prod_{j=0}^{\ell}{c_j}}{b}(A-B\widetilde{K}_{\ell+1})\Big) 
	<1. \label{eq13}
	\end{align}
	Based on (\ref{eq13}), we can find uniquely a positive definite symmetric  solution $\widetilde{P}_{\ell+1}$ to solve
		\begin{align*}
		&\frac{\prod_{j=0}^{\ell+1}{c^{2}_j}}{b^2}( A-B\widetilde{K}_{\ell+1} ) ^T\widetilde{P}_{\ell+1}( A-B\widetilde{K}_{\ell+1} )-\widetilde{P}_{\ell+1}+Q
		\nonumber\\
		&	+\widetilde{K}_{\ell+1}^{T}R\widetilde{K}_{\ell+1}=0.
	\end{align*}
	By mathematical induction, it can be proven that the given statements hold for every $i = 0, 1, 2,\ldots$. 
\end{proof}
\begin{rmk}	Lemma \ref{lemma2} can be seen as the scaling technique proposed in this paper.
	This Lemma  emphasizes the crucial role of term $ \prod_{j=0}^{i}{c_j}$   in scaling the original system (\ref{eq1}). Specifically, we refer to $\prod_{j=0}^{i}{c_j}$  as the cumulative factor and $c_i$  as the scaling factor. 
	In Lemma \ref{lemma2},  the choice of the constant $b$ is clearly dependent on both the system parameter matrices and the already determined initial control gain.
	In general, setting a larger value of $b$ will increase the stability of the system $(\frac{1}{b}A, \frac{1}{b}B)$, but this may require more iterations to fully compensate for its impact through the cumulative factor $\prod_{j=0}^{i}{c_j}$.

\end{rmk}


For convenience, we denote system $x_{k+1}= \frac{\prod_{j=0}^i{c_j}}{b}(Ax_k+Bu_k)$  as $\big(\frac{\prod_{j=0}^i{c_j}}{b}A, \frac{\prod_{j=0}^i{c_j}}{b}B\big)$. 
Lemma \ref{lemma2} provides a method to seek the stable control gain matrix of the original system  $(A, B)$. Based on (\ref{eq9})-(\ref{eq12}), we can build a sequence of system $\big(\frac{\prod_{j=0}^i{c_j}}{b}A, \frac{\prod_{j=0}^i{c_j}}{b}B\big)$ for $i=0,1,2,\ldots$. Based on Lemma \ref {lemma2}, we  then achieve a control gain sequence $\widetilde{K}_{i}$, which can stabilize the current system $\big(\frac{\prod_{j=0}^i{c_j}}{b}A, \frac{\prod_{j=0}^i{c_j}}{b}B\big)$  for $i=0,1,2,\ldots$.
Given an integer $\hat{i}$, if  $ \frac{\prod_{j=0}^{\hat{i}}{c_j}}{b}\geqslant1$ is satisfied, then 
 the closed-loop system (\ref{eq1}) can be stabilized by the control gain $\widetilde{K}_i$ for $i \geqslant \hat{i}$. 
 Therefore, the control gain $\widetilde{K}_i$ for $i\geqslant \hat{i}$ can serve as the initial  stabilizing control gain in Lemma \ref{lemma1}, which can be utilized to solve the optimal control problem. 
 The existence of $\hat{i}$ will be shown subsequently.

Drawing from the above discussion, a model-based  SPI algorithm has been devised for solving the optimal control problem utilizing the system information, as depicted in Algorithm \ref{alg:1}.


	From  Algorithm \ref{alg:1}, it is easy to see that the SPI algorithm designed in this paper consists of two main phases. 
In the first phase, we employ a scaling technique to achieve a stable control gain. 
Specifically, we start by arbitrarily selecting an initial  control gain $\widetilde{K}_{0}$, which doesn't need to stabilize the system. If \(\widetilde{K}_{0}\) is not stabilizing, we scale the original system parameter matrices such that the scaled system is Schur stable under this gain. 
Then, following our designed scaling principle, we progressively scale the system while ensuring that the scaled system remains Schur stable with the corresponding control gain generated throughout this process. This iteration is repeated until the final obtained control gain stabilizes the original system. 
In the second phase, utilizing the stabilized control gain initial PI  to obtain the  optimal control input and minimize performance index (\ref{eq2}). 
\begin{algorithm}[!h]
    \caption{Model-Based SPI Algorithm }\label{alg:1}
    \renewcommand{\algorithmicrequire}{\textbf{Initialize:}}
    \begin{algorithmic}[1]
        \REQUIRE Given any $\widetilde{K}_{0}\in \mathbb{R}^{m\times n}$,  set a prescribed small enough scalar $\mathcal{E}>0$, select $b$ such that $b>\rho(A-B\widetilde{K}_{0})$. Given iterative index $i=0$, $i_{max}$, and set $c_{0}=1$.
		\renewcommand{\algorithmicrequire}{\textbf{Iterative learning:}}
		\REQUIRE 
		\FOR{$i=0:i_{max}$}	
		\IF {$ \frac{\prod_{j=0}^{i}{c_j}}{b}<1$}
		\STATE Solve $\widetilde{P}_{i}$ from (\ref{eq10}) and update $\widetilde{K}_{i+1}$  by (\ref{eq11}).
		\STATE Determine $c_{i+1}$ satisfies (\ref{eq12}).
		 \ELSE
	\STATE $b \gets1$, $\prod_{j=0}^{i}{c_j} \gets1$.
	\STATE Solve $\widetilde{P}_{i}$ from (\ref{eq10}) and update $\widetilde{K}_{i+1}$  by (\ref{eq11}).
	\IF { $i\geqslant 1$ and $ ||\widetilde{P}_{i}-\widetilde{P}_{i-1}||<\mathcal{E}$}
	\STATE  \textbf{break} 
	\ENDIF
	\ENDIF
	\ENDFOR
	\RETURN $P^{*}\gets \widetilde{P}_{i};K^{*}\gets  \widetilde{K}_{i+1}$.
           \end{algorithmic}
\end{algorithm}
\begin{theorem} \label {theorem_algorithm_1_Convergences}
	The matrix series $ \{\widetilde{P}_i \}$ for $i=0,1,2,\dots, \hat{i},\ldots$, generated by Algorithm \ref{alg:1}, converges to the sole positive definite solution of the ARE (\ref{eq6}), that is, $\underset{i\rightarrow \infty}{\lim}\widetilde{P}_i=P^*$.
\end{theorem}
\begin{proof}
	Firstly, let we prove that the existence of $\hat{i}$, which is defined earlier. 
		For any given $\widetilde{K}_{0}$,  it can be inferred from (\ref{eq9}) that $b$ is a finite constant. 		
		If $\widetilde{K}_{0}$ is stable and $b$ is chosen such that $b < 1$, then $\hat{i} = 0$. However, if $b > 1$ or $\widetilde{K}_{0}$ is unstable, it follows from(\ref{eq12})  that 		
	  $c_{i} > 1$ for $i=1,2,\dots$, then  $c_{i}$ can be written as $c_{i}=1+\gamma_i$ with $\gamma _i>0$. Let $\gamma  =\min \left\{ \gamma _j,j=1,\dots ,i \right\} $, then one has 
	 \begin{align*}
		\frac{\prod_{j=1}^i{c_j}}{b}=\frac{1}{b}\prod_{j=1}^i{\left( 1+\gamma _j \right)}\geqslant \frac{1}{b}\left( 1+\gamma  \right) ^i.
	 \end{align*}
	 It is well known that the exponential function with a base greater than 1 is monotonically increasing and grows infinite, 
	  therefore, there  exists a positive integer $\hat{i}$, such that $\frac{1}{b}\left( 1+\gamma  \right) ^{\hat{i}}\geqslant1$, that is $\frac{\prod_{j=1}^{\hat{i}}{c_j}}{b}\geqslant1$.	 
	 This means that by iteratively performing Steps 3 and 4 of Algorithm \ref{alg:1}, we can obtain a control gain matrix $\widetilde{K}_{\hat{i}}$ that renders the matrix $A - B\widetilde{K}_{\hat{i}}$ is Schur stable. 
	  Commencing with the stable control gain matrix $\widetilde{K}_{\hat{i}}$, 
	 Algorithm \ref{alg:1} will be equivalent to the model-based PI algorithm outlined in Lemma \ref{lemma1}.
	 	Therefore  the convergence of  Algorithm \ref{alg:1} can be guaranteed.  The proof is finished.
\end{proof}
\begin{rmk}

		It is worth noting that the technique we employed is fundamentally different from that presented in \cite{chen2022homotopic}.
		For CT systems, stability is determined by Hurwitz criteria, requiring that all eigenvalues of the matrix $A + BK$ have negative real parts. Consequently, \cite{chen2022homotopic} employs a translation technique to achieve a stable control gain matrix.
		In contrast, for DT systems, stability is assessed using Schur stabilization, which requires the spectral radius of the matrix $A + BK$ to be less than $1$. This necessitates consideration of both the real and imaginary parts of the eigenvalues, making the analysis of DT systems inherently more complex. In this paper, we introduce a scaling technique that constructs a sequence of stable control systems, effectively converting any unstable control gain matrix into a stable one. This represents a novel contribution of our work.
\end{rmk}
\section{Model-free SPI algorithm design}
	\label{Sec_4}
	In this section, by employing data-driven techniques \cite{stability_2021}, we eliminate the assumption of complete knowledge of all system matrices and propose a model-free SPI algorithm based on data samples. This algorithm solves the optimal control problem for system (\ref{eq1}) while ensuring system stability.


	\subsection{ Model-free  SPI algorithm}
	

To develop the model-free SPI algorithm, one needs to utilize the trajectory of the system (1) as a foundation and has
	\begin{align*}
		&\frac{\prod_{j=0}^{i}{c_{j}^{2}}}{b^{2}}x_{k+1}^{T}\widetilde{P}_{i}x_{k+1}-x_{k}^{T}\widetilde{P}_{i}x_k 
		\nonumber\\
		=&x_{k}^{T}\Big[ \frac{\prod_{j=0}^{i}{c_{j}^{2}}}{b^{2}}( A-B\widetilde{K}_i ) ^T\widetilde{P}_{i}( A-B\widetilde{K}_i ) -\widetilde{P}_{i} \Big] x_k 
		\nonumber\\
		&+\frac{2\prod_{j=0}^{i}{c_{j}^{2}}}{b^{2}}x_{k}^{T}( A-B\widetilde{K}_i ) ^T\widetilde{P}_{i}B( \widetilde{K}_ix_k+u_k )
		\nonumber\\
		& +\frac{\prod_{j=0}^{i}{c_{j}^{2}}}{b^{2}}( \widetilde{K}_ix_k+u_k ) ^TB^T\widetilde{P}_{i}B( \widetilde{K}_ix_k+u_k )\nonumber\\
		=&-x_{k}^{T}( Q+\widetilde{K}_{i}^{T}RK_i ) x_k-\frac{\prod_{j=0}^{i}{c_{j}^{2}}}{b^{2}}x_{k}^{T}\widetilde{K}_{i}^{T}B^T\widetilde{P}_{i}B\widetilde{K}_ix_k
		\nonumber\\
		&+\frac{\prod_{j=0}^{i}{c_{j}^{2}}}{b^{2}}(2x_{k}^{T}A^T\widetilde{P}_{i}B( \widetilde{K}_ix_k+u_k )
		+u_{k}^{T}B^T\widetilde{P}_{i}Bu_k).
	\end{align*}
We write the above equation in the following form
\begin{align}\label{eq14}
		&\big( \frac{\prod_{j=0}^{i}{c_{j}^{2}}}{b^{2}}vecv(x_{k+1})-vecv(x_{k} )\big) vecs( \widetilde{P}_{i} )-\frac{\prod_{j=0}^{i}{c_{j}^{2}}}{b^{2}}
		\nonumber\\
		&\times\Big[2( x_{k}^{T}\otimes x_{k}^{T}\cdot \widetilde{K}_{i}^{T}\otimes I_n+u_{k}^{T}\otimes x_{k}^{T} )vec( A^T\widetilde{P}_{i}B ) \nonumber\\
		&-\left( vecv(K_{i}x_{k})-vecv(u_{k}) \right) vecs( B^T\widetilde{P}_{i}B )\Big]
		\nonumber\\
		&=-x_{k}^{T}\otimes x_{k}^{T}vec( Q+\widetilde{K}_{i}^{T}R\widetilde{K}_i ). 
	\end{align}
	Let the positive integer $l$ represent the number of data samples. For any given sequence of vectors $\{ z_k \} _{k=0}^{l}$, we define   
		\begin{align*}
		 & d_{z}=[vecv(z_{0}),\dots,vecv(z_{l-1})]^{T},
		 \nonumber\\
		 &D_{z}=[vecv(z_{1}),\dots,vecv(z_{l})]^{T},\nonumber\\
		&\delta _{ux}=\left[ u_0\otimes x_0,\dots,u_{l-1}\otimes x_{l-1} \right] ^T,
		\nonumber\\
		 &\delta _{xx}=\left[ x_0\otimes x_0,\dots,x_{l-1}\otimes x_{l-1} \right] ^T,\\
		&\varGamma _i=\delta _{xx}vec( Q+\widetilde{K}_{i}^{T}R\widetilde{K}_i ),\theta _i=\big[ \frac{\prod_{j=0}^{i}{c_{j}^{2}}}{b^{2}}D_{x}-d _{x},\\
		&\quad \quad-\frac{2\prod_{j=0}^{i}{c_{j}^{2}}}{b^{2}}(\delta _{xx}\widetilde{K}_{i}^{T}\otimes I_n+\delta _{ux}),		
		\frac{\prod_{j=0}^{i}{c_{j}^{2}}}{b^{2}}(d_{\widetilde{K}_{i}x}-d _{u}) \big], 
		\end{align*}
	where 
	 $\delta _{ux}\in \mathbb{R} ^{l\times mn}$,  $\delta _{xx}\in \mathbb{R} ^{l\times n^2}$, $\varGamma _i\in \mathbb{R} ^l$, $d_{x}\in \mathbb{R} ^{l\times\frac{n\left( n+1 \right)}{2}}$,  $D_{x}\in \mathbb{R} ^{l\times\frac{n\left( n+1 \right)}{2}}$, $d_{\widetilde{K}_{i}x}\in \mathbb{R} ^{l\times\frac{m\left( m+1 \right)}{2}}$,  $d_{u}\in \mathbb{R} ^{l\times\frac{m\left( m+1 \right)}{2}}$, $\theta _i \in \mathbb{R} ^{l\times[\frac{n\left( 1+n \right)}{2}+mn+\frac{m\left( 1+m \right)}{2}]}$, 
	 $ l\geqslant\frac{n\left( 1+n \right)}{2}+mn+\frac{m\left( 1+m \right)}{2}$.
		Let $M_{i}=A^T\widetilde{P}_{i}B$ and $L_{i}=B^T\widetilde{P}_{i}B$,  then  
		according to (\ref{eq14}) we have
		\begin{align}\label{eq15}
		\theta _i\left[ \begin{array}{c}
			vecs( P_{i} )\\
			vec( M_{i} )\\
			vecs( L_{i} )\\
		\end{array} \right] =-\varGamma _i
	\end{align}
	for $i=0,1,2,\ldots$.
		Select a sufficiently large $l$ such that 
	 the following condition is satisfied
	\begin{align}\label{eq16}
		rank\left( [\delta _{xx},\delta _{ux},d _{u}] \right) =\frac{n\left( 1+n \right)}{2}+mn+\frac{m\left( 1+m \right)}{2}.
	\end{align}
	Then, there exists a unique solution to equation (\ref{eq15}), which will be demonstrated later on.
	By employing the least squares approach, we determine the exclusive solution to equation (\ref{eq15}) as follows:
	\begin{align*}
	\left[ \begin{array}{c}
		vecs\left( P_{i} \right)\\
		vec\left( M_{i} \right)\\
		vecs\left( L_{i} \right)\\
	\end{array} \right] =-\left( \theta _{i}^{T}\theta _i \right) ^{-1}\theta _{i}^{T}\varGamma _i. 
\end{align*}

Based on lemma \ref{lemma2}, we achieve the control gain given as
	\begin{align}\label{eq17}
	\widetilde{K}_{i+1}=\Big( L_{i}+\frac{b^{2}}{\prod_{j=0}^{i}{c_{j}^{2}}}R \Big) ^{-1}M_{i}^{T}.
	\end{align}
	
	As can be seen from the derivation process,  parameters $b$ and $c_i$ are involved, which are determined by  requiring the  system knowledge. 
	The method for obtaining these parameters  when the system  matrices are unknown completely will be investigated in detail.
Prior to that, we present the subsequent lemma to establish the uniqueness of the solution to equation (\ref{eq15}).
\begin{lemma}\label{lemma3}
	 If the rank condition (\ref{eq16}) is satisfied,  then equation (\ref{eq15})  has a unique solution.
\end{lemma}
\begin{proof}
	To prove this result, it suffices to demonstrate that,  for matrices $S=S^{T}\in \mathbb{R}^{n\times n}$, $Y\in \mathbb{R}^{n\times m}$, $W=W^{T}\in \mathbb{R}^{m\times m}$, the following
	matrix equality 
	 \begin{align*}
		\theta _i\left[ \begin{array}{c}
			vecs\left( S \right)\\
			vec\left(Y \right)\\
			vecs\left(W \right)\\
		\end{array} \right] =0.
	\end{align*}
	holds  for   $i=0,1,2,\ldots$, if and only if 
	$S=0$, $Y=0$, and $W=0$.
		It is apparent that
		\begin{align*}
		\theta _i\left[ \begin{array}{c}
			vecs ( S )\\
			vec(Y )\\
			vecs(W )\\
		\end{array} \right] 
		=[d_{x},\delta _{ux},d_{u}]\left[ \begin{array}{c}
			vecs( G)\\
			vec(H )\\
			vecs(Z )\\
		\end{array} \right],
	\end{align*}
	where 
	\begin{align}
		G=&\frac{\prod_{j=0}^{i}{c_{j}^{2}}}{b^{2}}(A_{i}^{T}SA_{i}+A_{i}^{T}SB\widetilde{K}_{i}+\widetilde{K}_{i}^{T}B^{T}SA_{i}+\widetilde{K}_{i}^{T}W\widetilde{K}_{i}
		\nonumber\\
		&+\widetilde{K}_{i}^{T}B^{T}SB\widetilde{K}_{i}-Y\widetilde{K}_{i}-\widetilde{K}_{i}^{T}Y^{T})-S,\label{eq18}\\
		A_{i}=&A-B\widetilde{K}_{i},\nonumber\\
		H=&\frac{2\prod_{j=0}^{i}{c_{j}^{2}}}{b^{2}}(A_{i}^{T}SB+\widetilde{K}_{i}^{T}B^{T}SB-Y),\label{eq19}\\
		Z=&\frac{\prod_{j=0}^{i}{c_{j}^{2}}}{b^{2}}(B^{T}SB-W).\label{eq20}
	\end{align}
	Note that $G^{T}=G$ and $Z^{T}=Z$. 
	 If (\ref{eq16}) is satisfied, it becomes evident that $[d_{x},\delta _{ux},d_{u}]$ has full column rank.  Consequently, we are able to deduce that $G=0, H=0,Z=0$.

Combining  $\frac{\prod_{j=0}^{i}{c_{j}}}{b}>0$ with (\ref{eq19}) and (\ref{eq20}), we have 
	 \begin{align}
		&W=B^{T}SB, \label{eq21}\\
		&Y=A_{i}^{T}SB+\widetilde{K}_{i}^{T}B^{T}SB.\label{eq22}
	 \end{align}
	 	 Put these into (\ref{eq18}), we derive
	 \begin{align}\label{eq23}
		\frac{\prod_{j=0}^{i}{c_{j}^{2}}}{b^{2}}A_{i}^{T}SA_{i}-S=0.
	 \end{align}
	 If $b$ and $c_i$ satisfy the conditions (\ref{eq9}) and (\ref{eq12}) respectively,
	   it can be shown that $\frac{\prod_{j=0}^{i}{c_{j}}}{b}A_i$ is Schur stable for   $i=0,1,2,\ldots$. Moreover, the only solution to the equation (\ref{eq23}) is $S=0$. Subsequently, it follows from the equations (\ref{eq21}) and (\ref{eq22}) that $Y=0$ and $W=0$. This concludes the proof.
\end{proof}

	\begin{rmk}	
        To solve (\ref{eq15}), the probing noise $e$ must be incorporated into the control input and collect system data for $l$ moments,  with $l\geqslant\frac{n\left( 1+n \right)}{2}+mn+\frac{m\left( 1+m \right)}{2}$,  to ensure that the rank condition (\ref{eq16}) is satisfied \cite{zero_sum_1_DT}.  		
	\end{rmk}

  \subsection{Selection for $b$}


One possible way to determine the value of parameter $b$ in situations where the system dynamics are completely unknown is to gradually increase $b$ until condition  (\ref{eq9}) 
 is satisfied for a given  $\widetilde{K}_{0}$. In the following, we propose an evaluation criterion that can be used to verify whether $b$ satisfies condition (\ref{eq9}),  which is discussed in detail below.

We first select $\widetilde{K}_0$ and $b\geqslant 1$ arbitrarily, let $c_{0}=1$. Then, when the rank condition (\ref{eq16}) is satisfied, we can solve the matrix $\widetilde{P}_0$ uniquely from  equation (\ref{eq15}) with $i=0$.
If the positive definiteness of the matrix $\widetilde{P}_0$ is established, it implies that the system $(\frac{1}{b}A,\frac{1}{b}B)$ is stabilized by $\widetilde{K}_0$, and thus a constant $b$ satisfying  (\ref{eq9}) has been identified.
Otherwise, we recalculate $\widetilde{P}_0$ by replacing $b$ in equation (\ref{eq15}) with $b+\delta$, where  $\delta>0 $ is defined as a step size.  That is, we let
	\begin{align}\label{eq25}
	b \leftarrow b+\delta,
\end{align}
and repeat the above steps until we obtain a positive definite  $\widetilde{P}_0$.
	\subsection{Selection for  $c_{i}$}
	To implement our model-free SPI algorithm, we also need to provide a method for determining the scaling factor  $c_{i}$ for $i=1,2,\ldots$. Since the system dynamics are unknown, we can only rely on the collected samples to achieve this goal.
	
The following theorem indicates that when we follow the updating law (\ref{eq26}) to determine the scaling factor,
we then can achieve a stable control gain matrix  capable of stabilizing the corresponding scaled system. 
That is, the updating law (\ref{eq26})  can be used in (\ref{eq15}) and (\ref{eq17}) to perform the learning process to achieve the  
stabilizing control gain matrix for system (\ref{eq1}).

\begin{theorem}\label{theorem_model_free_c}
	Let the scaling factor  $c_{i+1}$ satisfy
\begin{align}\label{eq26}
	\begin{cases}
		c_{i+1}=1,&		if\,\,\mathcal{Q}_{i}\,\,is\,\,non-invertible\\
		1<c_{i+1}<\sigma (\widetilde{P}_{i}\mathcal{Q}_{i}^{-1})^{1/2},&		if \,\,\mathcal{Q}_{i} \,\,is\,\,  invertible\\
			\end{cases}
	\end{align}
	where $\mathcal{Q}_{i}=\widetilde{P}_{i}-Q-\widetilde{K}_{i+1}^{T}R\widetilde{K}_{i+1}$. 
	Then, using the control gain 
	$\widetilde{K}_{i+1}$ obtained from (\ref{eq11}), a unique positive definite solution $\widetilde{P}_{i+1}$ to the Lyapunov equation given by (\ref{eq10}) is guaranteed in the iteration step $i+1$ for $i=0,1,2,\ldots$.
\end{theorem}
\begin{proof}
	Given $b$ satisfies (\ref{eq9}) obtained by iteratively solving  (\ref{eq15}) with $i=0$ and $c_{0}=1$, and implementing (\ref{eq25}) until $\widetilde{P}_0>0$. 
	From (\ref{eq10}), (\ref{eq11}), and the proof of Lemma \ref{lemma2},
	 it can be deduced that  $\frac{1}{b}\prod_{j=0}^{i}{c_j}(A-B\widetilde{K}_{i+1})$ is Schur stable. 
	Therefore, by solving the Lyapunov equation below, we can derive the unique positive definite solution $\bar{P}_{i+1}$ for the equation
		\begin{align}\label{eq27}
	&\frac{1}{b^2}\prod_{j=0}^{i}{c^{2}_j}( A-B\widetilde{K}_{i+1}) ^{T}\bar{P}_{i+1}( A-B\widetilde{K}_{i+1}) -\bar{P}_{i+1}+Q\nonumber\\
	&+\widetilde{K}_{i+1}^{T}R\widetilde{K}_{i+1}=0. 
	\end{align}
According to Lemma {\ref{lemma1}}, we have 
	\begin{align}\label{P}
		0<\bar{P}_{i+1}\leq\widetilde{P}_{i},
	\end{align}
where $\widetilde{P}_{i}$ is solved from (\ref{eq10}).

Next, we will categorically discuss the stability of  $\frac{1}{b}\prod_{j=0}^{i+1}{c_j}\\
(A-B\widetilde{K}_{i+1})$ based on the reversibility of $\mathcal{Q}_{i}$.
\\ 
\textbf{Case 1:} 
$\mathcal{Q}_{i}$ is non-invertible

	Let $c_{i+1}=1$, according to the Lemma \ref{lemma1},  $\frac{1}{b}\prod_{j=0}^{i+1}{c_j}(A-B\widetilde{K}_{i+1})$ is also Schur stable.\\
\textbf{Case 2:}
$\mathcal{Q}_{i}$ is invertible

It follows from (\ref{eq26}), (\ref{eq27}) and (\ref{P}) that		
	\begin{align*}
		&\frac{1}{b^2}\prod_{j=0}^{i+1}{c^{2}_j}( A-B\widetilde{K}_{i+1}) ^{T}\bar{P}_{i+1}( A-B\widetilde{K}_{i+1}) -\bar{P}_{i+1}\\
		=&c^{2}_{i+1}\Big(\frac{1}{b^2}\prod_{j=0}^{\textcolor{blue}{i}}{c^{2}_j}( A-B\widetilde{K}_{i+1}) ^{T}\bar{P}_{i+1}( A-B\widetilde{K}_{i+1}) -\bar{P}_{i+1}\Big)
		\\
		&-(1-c^{2}_{i+1})\bar{P}_{i+1}\\
        =&-c^{2}_{i+1}(Q+\widetilde{K}_{i+1}^{T}R\widetilde{K}_{i+1})-(1-c^{2}_{i+1})\bar{P}_{i+1}\\
		\leqslant &-c^{2}_{i+1}(Q+\widetilde{K}_{i+1}^{T}R\widetilde{K}_{i+1})-(1-c^{2}_{i+1})\widetilde{P}_{i}\\
		=&c^{2}_{i+1}(\widetilde{P}_{i}-Q-\widetilde{K}_{i+1}^{T}R\widetilde{K}_{i+1})-\widetilde{P}_{i}\\
		<&\sigma\big(\widetilde{P}_{i} \big(\widetilde{P}_{i}-Q-\widetilde{K}_{i+1}^{T}R\widetilde{K}_{i+1}\big)^{-1}\big)I(\widetilde{P}_{i}-Q-\widetilde{K}_{i+1}^{T}R\widetilde{K}_{i+1})-\widetilde{P}_{i}\\
		\leqslant&\widetilde{P}_{i} \big(\widetilde{P}_{i}-Q-\widetilde{K}_{i+1}^{T}R\widetilde{K}_{i+1}\big)^{-1}\big(\widetilde{P}_{i}-Q-\widetilde{K}_{i+1}^{T}R\widetilde{K}_{i+1}\big)-\widetilde{P}_{i}\\
		=&0,
	\end{align*}
	which reveals that $\frac{1}{b}\prod_{j=0}^{i+1}{c_j}(A-B\widetilde{K}_{i+1})$ is Schur stable. 

	Thus, a unique  positive definite solution $\widetilde{P}_{i+1}$ is established for equation (\ref{eq10}) at the $(i+1)$-th iteration step. The proof is  concluded.
\end{proof}
	As demonstrated in Theorem \ref{theorem_model_free_c}, it is clear that \( c_{i+1} \geq 1 \), and therefore, \( \frac{\prod_{j=0}^i c_j}{b} \) is monotonically non-decreasing as \( i \) increases. Moreover, since \( \mathcal{Q}_{i} \) is not always irreversible, the value of \( c_{i+1} \) does not remain fixed at 1. Consequently, we can conclude that after a finite number of learning iterations, there exists a constant \( \hat{i} \) such that \( \frac{\prod_{j=0}^{\hat{i}} c_j}{b} \geq 1 \).
\begin{algorithm}[!h]
    \caption{Model-Free SPI Algorithm }\label{alg:2}
    \renewcommand{\algorithmicrequire}{\textbf{Initialize:}}
    \begin{algorithmic}[1]
        \REQUIRE Choose $\delta>0$, $\mathcal{E}>0$, $b\geqslant1$. Give iterative index $i=0$, $i_{max}$, set $\widetilde{P}_{0}=-I$, $c_{0}=1$. 
			Employ a measurable locally essentially bounded  control input $u_k=\widetilde{K}_0x_k+e$ to system (\ref{eq1}), where $\widetilde{K}_{0}$ is a arbitrary given control gain, $e$ denotes the exploration noise.
			\renewcommand{\algorithmicrequire}{\textbf{Data Collection:}}
           \REQUIRE Collect the online data of system state $x_{k}$ and input $u_{k}$, for $k=0,2,\cdots, l$, compute $\delta _{ux},\delta _{xx}, d_x,D_x,d_u$, 
		   where $l$ is such that (\ref{eq16})  holds.
		\renewcommand{\algorithmicrequire}{\textbf{Iterative learning:}}
		\REQUIRE 
		\WHILE{$i<i_{max}$}	
		\IF {$ \widetilde{P}_{0}> 0$}
\STATE Compute $\theta _{i}$ and $\varGamma _i$.
\STATE Solve $\widetilde{P}_{i}$ from (\ref{eq15}) and  update $\widetilde{K}_{i+1}$ by (\ref{eq17}).
\STATE Determine $c_{i+1}$.
\STATE $i \gets i+1$.
		\ELSE
		\STATE $b \gets b+\delta $ 
		\STATE Compute $\theta _{0}$ and $\varGamma _0$.
    \STATE Solve $\widetilde{P}_{0}$ from (\ref{eq15}).
    			\ENDIF
		\IF {$ \frac{\prod_{j=0}^{i}{c_j}}{b}\geqslant1$}
		\STATE  \textbf{break.} 
		\ENDIF
			\ENDWHILE 
	\WHILE{$i<i_{max}$}	   
	\STATE $b \gets1$, $\prod_{j=0}^{i}{c_j} \gets1$.
	\STATE Solve $\widetilde{P}_{i}$ from (\ref{eq15}) and  update $\widetilde{K}_{i+1}$ by (\ref{eq17}).
		\IF { $||\widetilde{P}_{i}-\widetilde{P}_{i-1}||<\mathcal{E}$}
	\STATE  \textbf{break} 
	\ENDIF
	\STATE $i\gets i+1$.
	\ENDWHILE
	\RETURN $P^{*}\gets \widetilde{P}_{i};K^{*}\gets  \widetilde{K}_{i+1}$.
           \end{algorithmic}
\end{algorithm}
\subsection{SPI algorithm synthesis}
Based on the above derivation, we can now develop the following model-free SPI algorithm, as listed in Algorithm \ref{alg:2}.

As shown in Algorithm \ref{alg:2}, compared to the conventional model-free PI algorithm, the proposed SPI algorithm has the advantage of iteratively solving for the optimal control gain matrix using only system dynamics data, without the need for a predetermined stable initial control gain matrix.  
The subsequent theorem ensures the convergence of the developed model-free SPI algorithm.
\begin{theorem}
Sequences $\{ \widetilde{P}_i \} $ and $\{ \widetilde{K}_i \}$  learned by the model-free SPI  algorithm  converge to $P^{*}$ and $K^{*}$, respectively.
\begin{proof}
When condition (\ref{eq16}) is met, one can ensure that (\ref{eq15}) has a unique solution  based on Lemma \ref{lemma3}. 
 Consequently, in the first loop of Algorithm 2, by repeatedly executing Steps 8-10, a positive definite matrix $\widetilde{P}_{0}$ is obtained, indicating that a $b$ satisfying (\ref{eq9}) has been found. Then, by iterating through Steps 3-6, a stable control gain matrix $\widetilde{K}_{\hat{i}}$ can be derived.
 Let $ \frac{\prod_{j=0}^{i}{c_j}}{b}=1$ for $i=\hat{i}, \hat{i}+1, \hat{i}+2, \dots$. In this case, based on the definitions of $M_i$ and $L_i$, it can be inferred that the sequences $\{ \widetilde{P}_i \}$ and $\{ \widetilde{K}_i \} $ obtained through Steps 16-23 of Algorithm \ref{alg:2} are equivalent to those derived from the model-based PI algorithm described in Lemma \ref{lemma1}, given the same initial stabilizing control gain $\widetilde{K}_{\hat{i}}$.
  Therefore, we can conclude that the convergence of the sequences $\{ \widetilde{P}_i \}$ and $\{ \widetilde{K}_i \}$ obtained through Algorithm \ref{alg:2} is guaranteed. This completes the proof.
     \end{proof}
\end{theorem}
\begin{rmk}
		Unlike Algorithm \ref{alg:1}, Algorithm \ref{alg:2} relies entirely on system sample data to determine the control gain matrix. For any given $\tilde{K}_{0}$, after finding an appropriate $b$, the parameters $B^T\widetilde{P}_{i}B$ and $B^T\widetilde{P}_{i}A$ (i.e., $M_i$ and $L_i$) required to update $\tilde{K}_{i+1}$ are solved using (\ref{eq15}), and then substituted into (\ref{eq17}) to obtain $\tilde{K}_{i+1}$, completely eliminating the need for knowledge of matrices $A$ and $B$.	
\end{rmk}
\begin{rmk} \label{Remark_comparison}
  The merits of model-free  SPI algorithm proposed in this paper compared to some of the existing works are as follows. In comparison to PI methods such as those proposed in \cite{ zero_sum_1_DT,lai2023model}, our algorithm eliminates the necessity for an initial stabilizing control policy. Unlike the GPI in \cite{ZhouGpI}, the execution of our model-free SPI algorithm circumvents the requirement for the maximum eigenvalue of the optimal value matrix $P^{*}$. Obtaining this eigenvalue is challenging without complete knowledge of the system matrices. Diverging from approaches in \cite{lamperski2020computing,lai2023learning} which need for a search procedure for the discount factor, our model-free SPI algorithm only requires a single iteration to determine the scaling factor that satisfies the specified condition. 
Compared to the LMI method in \cite{de2019formulas} and the deadbeat control gain matrix design method in \cite{van2020data, RL_2}, the approach in this paper offers greater versatility. The scaling technique can be directly extended to DT linear systems where \(P_i\) exhibits non-decreasing monotonicity during PI. In contrast, the methods in \cite{de2019formulas} and \cite{van2020data} cannot be directly applied to different linear systems, such as the stochastic systems with multiplicative noise discussed in \cite{wang2016infinite}, is challenging and requires further research.


\end{rmk}

\section{Simulation results}\label{Sec_5}
	In this section,  a numerical example is presented to evaluate  the proposed algorithm. We
	consider a power system as described in \cite{2016Asymptotically},  which takes the form of $	\dot{x}\left( t \right) =A_cx\left( t \right) +B_cu\left( t \right) ,$
	where 
\begin{align*}
	&	x=[\Delta \bar{\alpha},\Delta P_{m},\Delta f_{G}]^{T},~~~u=\Delta P_{c},\\
	&A_{c}=\left[\begin{array}{ccc}
		-\frac{1}{T_{g}} & 0 & \frac{1}{R_{g} T_{g}} \\
		\frac{K_{t}}{T_{t}} & -\frac{1}{T_{t}} & 0 \\
		0 & \frac{K_{p}}{T_{p}} & -\frac{1}{T_{p}}
		\end{array}\right],~~~
		B_{c}=\left[\begin{array}{c}
			\frac{1}{T_{g}} \\
			0 \\
			0
			\end{array}\right].
	\end{align*}
	In this power system, $\Delta \bar{\alpha}$ represents the incremental change in the position of the governor value,  $\Delta P_{m}$ represents the incremental change in the output of the generator.
	 $\Delta f_{G}$ denotes the incremental frequency deviation, $\Delta P_{c}$ denotes the incremental change in speed for position deviation,  $T_{g}$ and $T_{t}$ represent the governor time and turbine time, respectively. $T_{p}$ is the generator model time, and $R_{g}$ is the feedback regulation constant, $K_{p}$ and $K_{t}$ denote the gain constant of generator model and turbine model, respectively.
	Similar to \cite{2016Asymptotically}, we select $T_{g}=0.08 s$, $T_{t}=0.1 s$, $T_{p}=20 s$, $R_{g}=2.5 Hz/MW$, $K_{p}=120 Hz/MW$, $K_{t}=1 s$. 
	
	Discretizing the above system  by the zero-order hold method with the sampling interval $T=0.01 s$ leads to 
	\begin{align}\label{eq29}
	x_{k+1}=Ax_{k}+Bu_{k},
	\end{align}
	with
	\begin{align*}
		A=\left[ \begin{matrix}
			0.8825&		0.0014&		0.0470\\
			0.0894&		0.9049&		0.0023\\
			0.0028&		0.0571&		0.9995\\
		\end{matrix} \right],
		B=\left[ \begin{array}{c}
			0.0001\\
			0.1190\\
			0.0036\\
		\end{array} \right]. 
	\end{align*}
	Select matrices $Q=I_{3}$ and $R=1$  for (\ref{eq2}).
To demonstrate the effectiveness of SPI algorithm proposed in this paper, we select $\widetilde{K}_{0}=0$  as the starting control gain  of our proposed algorithm, and  set the convergence tolerance to be $\mathcal{E}=10^{-5}$. 

In this scenario,  the eigenvalues of  matrix $(A-B\widetilde{K}_{0})$ are $0.8847 \pm 0.0405i, 1.0176$,  it is evident that the initial control gain $\widetilde{K}_{0}$ is  unstable.  Using our proposed algorithm,  we can see that, after $10$ iterations by Algorithm \ref{alg:1},  the optimal value  matrix $P^{*}$ and control gain $K^{*}$ have been achieved as follows:
\begin{align*}
	P^{*}=&\left[ \begin{matrix}
		6.4599&		\,\,3.2440&		6.3364\\
		\,\,3.2440&		7.6499&		\,\, 10.1346\\
		6.3364&		\,\, 10.1346&		\,\,33.5195\\
	\end{matrix} \right] ,\\
	K^{*}=&\left[ \begin{matrix}
		0.4022&		\,\,0.8351&		\,\,1.2066\\
	\end{matrix} \right],
\end{align*}
which are consistent with the results obtained by the conventional model-based  PI algorithm based on Lemma \ref{lemma1}.
The convergences of $\widetilde{P}_{i}$ and $\widetilde{K}_{i}$ during this process are depicted in Fig.\ref{fig1}. 

Now, assuming a complete lack of information regarding the dynamical system matrices, we employ  the model free SPI algorithm to solve the optimal control problem.  
Set the initial state as $x_{0}=[0.1,0.1,0.2]^T$, apply the control input $u_{k}=\sum_{h=1}^{100}{\sin \left( \omega _{h}k \right)}$ to the system (\ref{eq29}), and collect the system data  $\delta _{ux},\delta _{xx}, d_x,D_x,d_u$ and system state for $k=0,1,2,\dots,30$, where $\omega _h$ is a random number uniformly distributed in the range $[-10,10]$.
Subsequently, the collected system data is repeatedly utilized to compute $\theta _i$ and $\varGamma _i$, thereby updating  (\ref{eq15}) and implementing the iterative  process of Algorithm \ref{alg:2}. 
We set $\delta=0.1$ and $b=1$, then the optimal value function matrix $P^{*}$ and control gain $K^{*}$ are acquired after 9 iterations using   model-free SPI Algorithm \ref{alg:2}. The convergence trajectories of $\widetilde{P}_{i}$ and $\widetilde{K}_{i}$ throughout this process are illustrated in Fig.\ref{fig2}. 
\begin{figure}[htbp]
	\centering
	\begin{subfigure}{0.48\linewidth}
	  \centering
	  \includegraphics[width=1\linewidth]{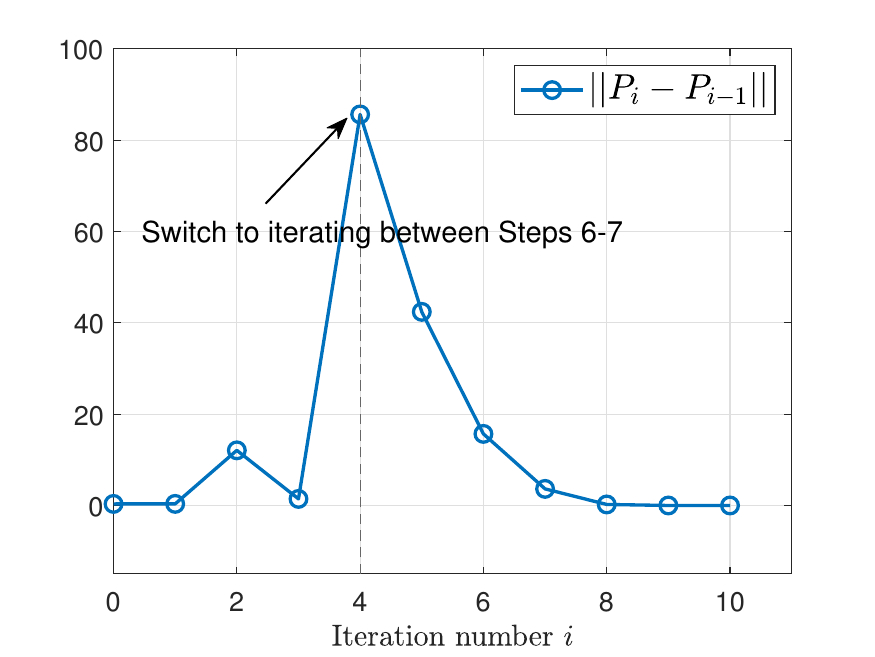}
	  \caption{}
	  \label{fig1_1}
	\end{subfigure}
	\begin{subfigure}{0.48\linewidth}
	  \centering
	  \includegraphics[width=1\linewidth]{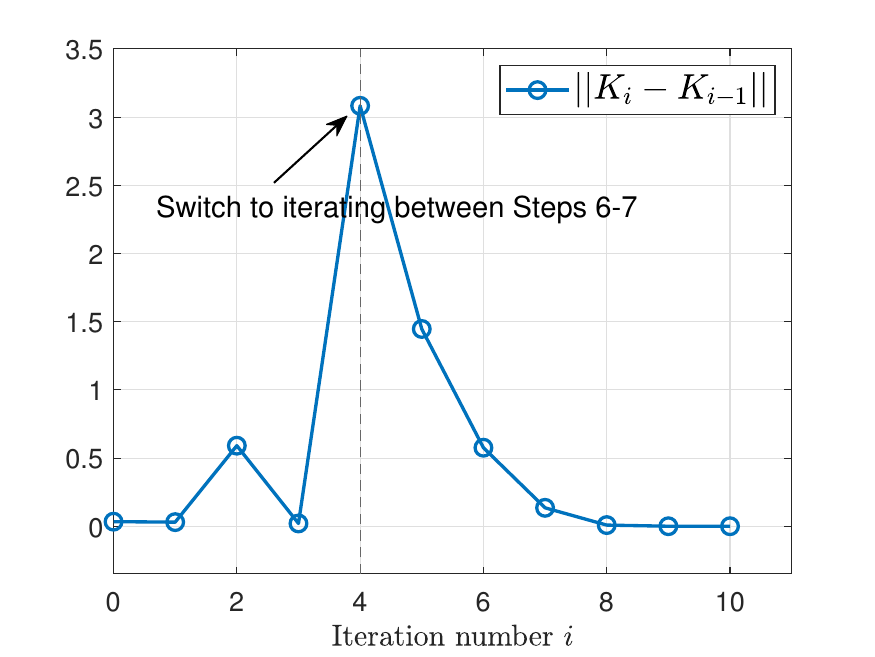}
	  \caption{}
	  \label{fig1_2}
	\end{subfigure}	
		\caption{Convergences of $\widetilde{P}_{i}$ and $\widetilde{K}_{i}$ in Algorithm \ref{alg:1}
		}
	\label{fig1}
\end{figure}
\begin{figure}[htbp]
	\centering
	\begin{subfigure}{0.48\linewidth}
	  \centering
	  \includegraphics[width=1\linewidth]{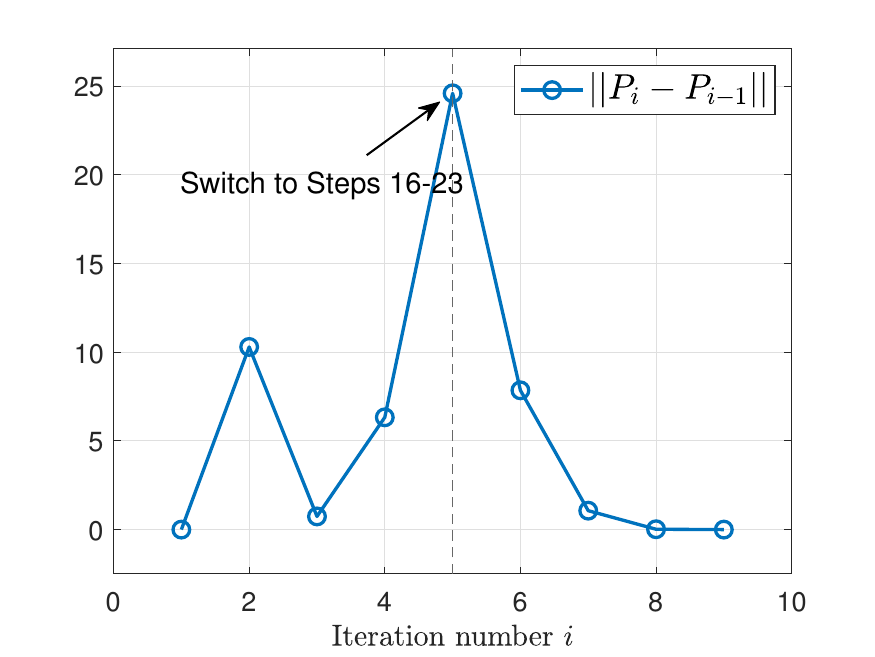}
	  \caption{}
	  \label{fig2_1}
	\end{subfigure}
	\begin{subfigure}{0.48\linewidth}
	  \centering
	  \includegraphics[width=1\linewidth]{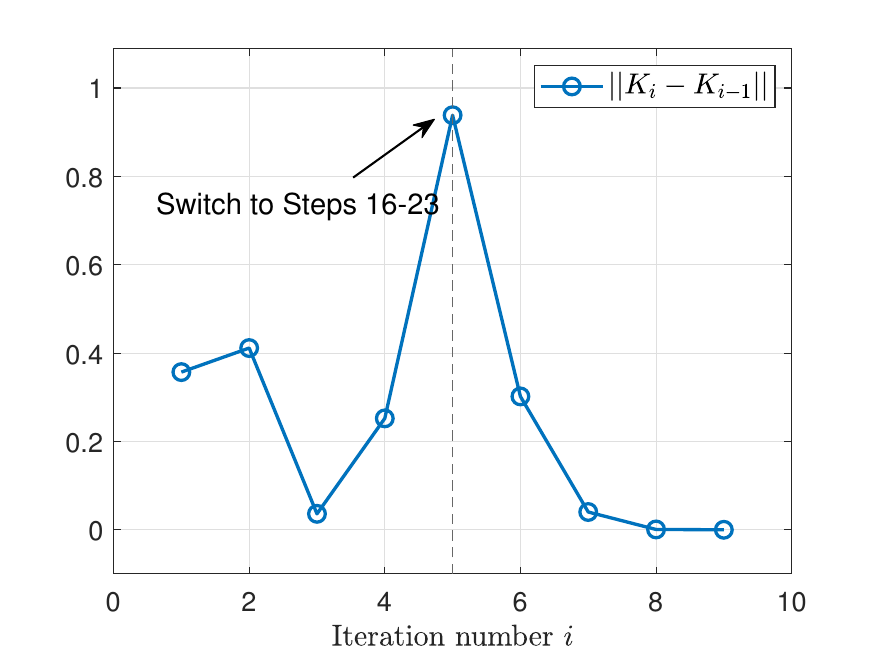}
	  \caption{}
	  \label{fig2_2}
	\end{subfigure}	
		\caption{Convergences of $\widetilde{P}_{i}$ and $\widetilde{K}_{i}$ in Algorithm \ref{alg:2}}
	\label{fig2}
\end{figure}

As mentioned before, the merit of the  proposed SPI algorithm over the conventional PI is that, instead of  requiring an initial stable control gain, the proposed SPI  algorithm can generate a stable control gain based on any given initial control gain. 
To highlight this advantage, we offer a detailed exposition of the process for obtaining the stable control gain in the algorithm.


  In Algorithm 1, with $ b = 2.0176 $, the changes in the relevant parameter values during the iterative process between Steps 3 and 4 are summarized in Table \ref{table01}. As shown in the table, the stopping condition \(\frac{\prod_{j=0}^{i}{c_{j}}}{b}\geqslant1\) is met after $4$ iterations, and the stable control gain $\widetilde{K}_{4}=[0.1829, 0.4622, 0.3963]$ is obtained with $\rho(A-B\widetilde{K}_4)=0.9928 < 1$. Furthermore, Table \ref{table01} also indicates that in each iteration, the parameter $ c_{i} $ satisfies (\ref{eq12}), and $\rho (\frac{\prod_{j=0}^{i}{c_{j}}}{b}A_{i})$ remains Schur stable. Subsequently, the iterative process between Steps 6 and 7 in Algorithm \ref{alg:1} is performed, and as shown in Fig.\ref{fig1}, after 6 iterations of learning, the optimal solution $ P^{*} $ is obtained.

	\begin{table}[!htbp] 
		\caption{
			Evolutions of the corresponding parameters during iterations between Steps 3 and 4 in Algorithm \ref{alg:1}.
			} 
		\label{table01}
		\centering
		\scriptsize
		\begin{tabular}{ccccccccc} 
		\toprule 
		$i$&$\rho (A_i)$&$\rho (\frac{\prod_{j=0}^{i}{c_{j}}}{b}A_{i})$&$\rho ^{-1}(\frac{\prod_{j=0}^{i}{c_{j}}}{b}A_{i+1})$&$\frac{\prod_{j=0}^{i}{c_{j}}}{b}$&$c _i$\\  
		\hline 
		 0&1.0176 &0.5044&1.9840&  0.4956 & 1\\   
		1& 1.0169 &0.5040&1.5674&0.6278&1.2666\\
		2&1.0163& 0.6380&1.0635 &0.9469&1.5084\\
		3&0.9930 & 0.9403&1.0498&0.9595&1.0133\\
		4&0.9928&0.9525& &1.0056&1.0480\\
		\bottomrule 
		\end{tabular}
				\end{table}

For model-free SPI algorithm,  Steps $1-15$ of Algorithm \ref{alg:2}  is executed by initializing $\widetilde{K}_0 = 0$, $b = 1$, $\delta=0.1$, and after $1$ iterations of learning, the value of $b=1.1$ that makes the condition $\widetilde{P}_{0}>0$ hold true has been found without relying on knowledge of the system matrices.

Then, as shown in Table \ref{table02}, 
after 5 iterations of Steps $3-6$ in the learning process, we observed that  $\frac{\prod_{j=0}^{4}{c_{j}}}{b}=1.0096 > 1$ holds true, which indicates that the stable control gain of the original system has been obtained given as $\widetilde{K}_{4}=[0.2649,0.6001,0.6767]$ and $\rho (A-B\widetilde{K}_4)=0.9790<1$. 
Moreover, 
Table \ref{table02} clearly reflects  the selection of $c_{i}$, as well as the corresponding evolution of 
$\sigma ( \mathcal{Q} _i ) $ and 
$\sigma (\widetilde{P}_{i}\mathcal{Q}_{i}^{-1})^{1/2}$,
 throughout the entire iterative process of the Steps $1-15$ of Algorithm \ref{alg:2}. It is evident that 
  $\mathcal{Q} _i$ is always reversible and $c_{i}$ consistently remains smaller than  
 $\sigma (\widetilde{P}_{i}\mathcal{Q}_{i}^{-1})^{1/2}$ 
 during the iterations.

\begin{table}[!htbp] 
	\caption{Evolutions of the corresponding parameters during iterations between Steps $3-6$ in Algorithm \ref{alg:2}.	
	}
		\label{table02}
	\centering
	\scriptsize
	\begin{tabular}{ccccccccc} 
	\toprule 
	$i$&$\rho (A_i)$&$\frac{1}{b}\prod_{j=0}^{i}{c_{j}}$&
	$\sigma (\widetilde{P}_{i}\mathcal{Q}_{i}^{-1})^{1/2}$
	&$c _i$& 
	$\sigma \left( \mathcal{Q} _i \right) $\\  
	\hline 
	 0&1.0176 &0.9091&1.0862& 1 & 1.4659\\   
	1&  1.0059 &0.9618&1.0415&1.0580 &1.9066 \\
	2& 0.9881&  0.9646&1.0420 &1.0029&1.9334\\
	3& 0.9897 & 0.9770&1.0356 &1.0129 &2.0826\\
	4&0.9790&1.0096&&1.0333 &\\
	\bottomrule 
	\end{tabular}
			\end{table}
  \begin{figure}
	\centering
	\includegraphics[width=0.42\textwidth]{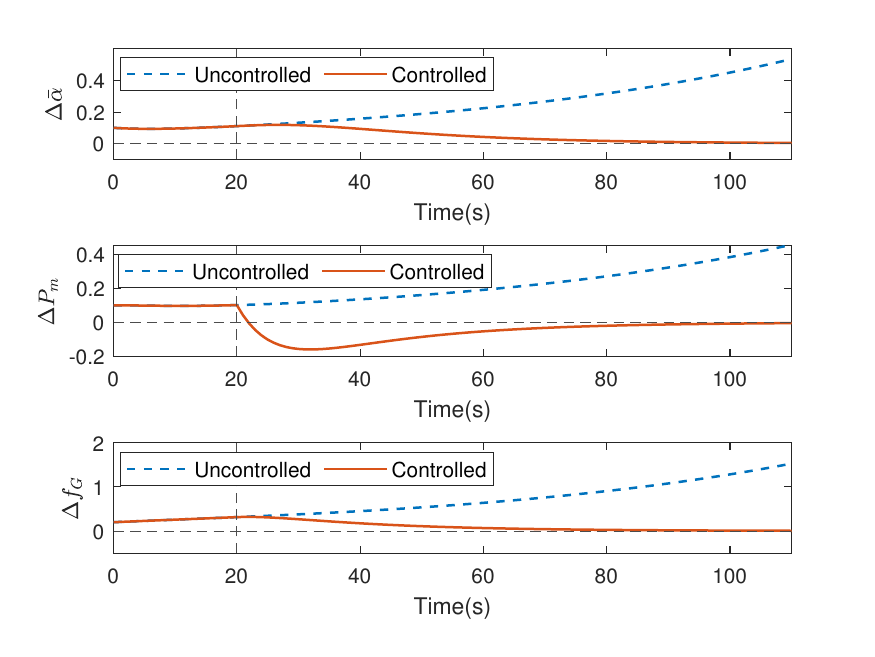}
	\caption{Trajectories of the system with and without a  controller}
	\label{fig:5}
\end{figure}
Given $\widetilde{K}_{4}$ as the stable control gain, Steps $16-23$ of Algorithm \ref{alg:2} are implemented to ascertain the optimal solution.  As illustrated in Fig. \ref{fig2}, after $5$ iterations, the optimal  $P^{*}$ and $K^{*}$ are achieved.

The state trajectories of the discretized power system (\ref{eq29}) under the controller designed by Algorithm \ref{alg:2} are depicted in Fig. \ref{fig:5}.  In these system trajectories, no control input is exerted on the system during the first 20 seconds.
To facilitate comparison, the state trajectories of the system with zero control input are also depicted in Fig. \ref{fig:5}. It is evident that, upon implementing the proposed control policy developed in Algorithm \ref{alg:2},  all system states converge to zero.


			



Next, we compare the model-free SPI algorithm  with the conventional VI   in \cite{VI} and  HI algorithm presented in \cite{HI2023adaptive}, and the data-based Q-learning algorithm (DQL) in \cite{RL_2}. These three algorithms represent the most classical, currently most prevalent, and relatively novel approaches  for solving the optimal control problem
without necessitating an initial stabilizing control gain, respectively. 

 Given the initialization requirements for the three algorithms: VI  requires ${P}_{0}$, HI  requires both ${P}_{0}$ and $\hat{Q}$, and SPI   requires $\widetilde{K}_0$. We therefore randomly generated 100 instances of ${P}_{0}$, set $\widetilde{K}_0=(R+B^{T}{P}_0B)^{-1}B^{T}{P}_{0}A$ to apply these four algorithms for solving the optimal control problem, respectively.
To ensure fairness, we accounted for the iteration required to determine the parameter \( b \) within the SPI process. 
Additionally, we defined  $\hat{Q} = Q+I_3$, $b=1$ and $\delta=0.7i$, where $i$ denotes the number of iterations, and $||\widetilde{K}_{i} - K^{*}|| < 10^{-4}$ was used as the convergence criterion for all four algorithms. 
The average running time and number of iterations required for convergence are  illustrated in Table \ref{table03}. 

Clearly, in the simulation of this power system, the SPI algorithm requires fewer iterations to converge to $K^{*}$ compared to the other three algorithms. It is also noteworthy that DQL takes the least amount of time, due to its model-free design, which requires less data and allows for faster equation solving. 
However, as noted in Remark \ref{Remark_comparison}, the scaling technique proposed in this paper offers broader applicability compared to the method of designing stable control gain matrix in the DQL algorithm.
Therefore, combining SPI with the Q-learning method from \cite{RL_2} to reduce runtime and efficiently address other control problems would be a valuable extension of this research.
 All the simulations were  run on a 12-core Intel i7-12650H 2.30 GHz CPU with 16-GB RAM and in MATLAB R2021b; no GPU computing was utilized.

\begin{table}[!htbp] 
	\caption{Performance comparison of the model-free SPI algorithm and VI, HI, and DQL.
		}
		\label{table03}
	\centering
	\scriptsize
	\begin{tabular}{ccccccccc} 
	\toprule 
	Algorithm& SPI& VI&
	HI
	&DQL\\  
	\hline 
	run-time (s)&0.0017 &0.0063& 0.0011& 0.0006 \\
	No. of iterations& 10 &116 &18 &13  \\
		\bottomrule 
	\end{tabular}
			\end{table}

\section{Conclusion} \label{sec_6}

In this paper, we have  proposed a scaling PI technique for DT systems. Based on this technique, two algorithms, namely model-based SPI  and model-free SPI algorithms, are introduced to successfully solve the optimal control problem for DT systems. It is noteworthy that these algorithms can obtain a stabilizing control gain by scaling a stable system sequence towards the original system.
This obviates the requirement for an initial stabilizing control gain in the PI.  Finally, we applied the algorithms to a power system and showcased their effectiveness, indicating their practical applicability in efficiently controlling power systems. 
 Theoretical findings and proposed algorithms in this paper are believed to have significant practical value for optimal control problems.
  In  future studies, we will extend the obtained results to deal with DT $H_\infty$ control problems, output feedback regulation, and optimal control problems for Markovian jump systems. In addition, considering the highly nonlinear nature of the real world,
   applying the proposed algorithm to nonlinear systems will also be a direction of our future research.


\setcounter{secnumdepth}{0} 
%
%
\addcontentsline{toc}{section}{References}



\bibliographystyle{apalike2}

\bibliography{cas-refs}

\end{document}